\documentclass[11pt,reqno]{amsart}

\usepackage{hyperref}
\newtheorem{theorem}{Theorem}[section]
\newtheorem{lemma}[theorem]{Lemma}
\newtheorem{corol}[theorem]{Corollary}
\numberwithin{equation}{section}

\theoremstyle{remark}
\newtheorem{remk}[theorem]{Remark}

\theoremstyle{definition}
\newtheorem{defin}[theorem]{Definition}

\usepackage{mathrsfs,bbm,amssymb}

\usepackage[all]{xy}
\usepackage[figuresright]{rotating}

\def\al{\alpha} 		\def\be{\beta}   
\def\eps{\varepsilon}		\def\ga{\gamma}   
 		\def\vi{\varphi}   
 \def\la{\lambda}	
 \def\si{\sigma}		

\def\Ga{\varGamma}     \def\Om{\Omega}		\def\Si{\varSigma}
\def\De{\Delta}				\def\La{\Lambda}

\def\mZ{\mathbb Z}		
\def\mQ{\mathbb Q}		
		\def\mP{\mathbb P}
\def\mD{\mathbb D} 		 

\def\aK{\mathbbm k}

\def\ba{\bar{a}}				\def\bb{\bar{b}}

		\def\bP{\mathbf P}

		\def\gM{\mathfrak m}

\def\kT{\mathcal T}		\def\kQ{\mathcal Q} 
\def\kR{\mathcal R} 

\def\hH{\hat{H}}	\def\hM{\hat{M}}	
		\def\hR{\hat{R}}
\def\hA{\hat{A}}

	\def\Tf{\tilde{f}}
\def\tV{\tilde{V}}	\def\tM{M^\sharp}
	\def\tR{R^\sharp}
\def\htR{\hat{R}^\sharp}	\def\tN{N^\sharp}

\def\oM{\bar{M}}		\def\oN{\bar{N}}
\def\oS{\bar{S}}		\def\oT{\bar{T}}
		\def\oal{\bar{\al}}
\def\obe{\bar{\be}}	\def\oL{\bar{L}}
		\def\oC{\bar{C}}
\def\oxi{\bar{\xi}}

\def\dd{\partial}
\def\0{\emptyset}	

\def\lat{\mbox{-}\mathrm{lat}}

\def\Hom{\mathop\mathrm{Hom}\nolimits}
\def\aut{\mathop\mathrm{Aut}\nolimits}
\def\Dim{\mathop\mathbf{dim}\nolimits}
\def\rad{\mathop\mathrm{rad}\nolimits}
\def\cok{\mathop\mathrm{Coker}\nolimits}

\def\Ch{\mathop\mathrm{Ch}\nolimits}
\def\Cr{\mathop\mathrm{Cr}\nolimits}

\def\End{\mathop\mathrm{End}\nolimits}
\def\Aut{\mathop\mathrm{Aut}\nolimits}

\def\GL{\mathop\mathrm{GL}\nolimits}

\def\Mat{\mathop\mathrm{Mat}\nolimits}
\def\im{\mathop\mathrm{Im}\nolimits}
\def\ker{\mathop\mathrm{Ker}\nolimits}
\def\rep{\mathop\mathrm{rep}\nolimits}

\def\lat{\mbox{-}\mathrm{lat}}

\def\mtr#1{\begin{pmatrix}#1\end{pmatrix}}
\def\smtr#1{\left(\begin{smallmatrix}#1\end{smallmatrix}\right)}

\def\setsuch#1#2{\left\{\,#1\mid #2\,\right\}}

\def\*{\otimes}		\def\+{\oplus}	\def\xx{\times}
\def\8{\infty}
\def\sb{\subset}         \def\sp{\supset}
\def\spe{\supseteq}      \def\sbe{\subseteq}
\def\bop{\bigoplus}

\def\bul{{\scriptscriptstyle\bullet}}

\def\ito{\stackrel{\sim}{\to}}
\def\xarr{\xrightarrow}
\def\ito{\stackrel{\sim}{\to}}

\newcommand{\rp}[5]{\fbox{$\scriptstyle #1\,\begin{smallmatrix}#2\\#3\\#4\\#5\end{smallmatrix}$}}

\def\iff{if and only if }
\def\AR{Auslander--Reiten}

\def\crg{crystallographic group}
\def\chg{Chernikov group}
\def\oc{one-to-one correspondence}
\def\wrt{with respect to }

\begin{document}

\title[Cohomologies of regular lattices]{Cohomologies of regular lattices \\ over the Kleinian 4-group}
\author{Yuriy Drozd}
\author{Andriana Plakosh}
\address{Institute of Mathematics\\ National Academy of Sciences of Ukraine \\
Tereschenkivska 3\\ 01024 Kyiv\\ Ukraine}
\email{y.a.drozd@gmail.com, drozd@imath.kiev.ua \emph{(corresponding author)}}
\email{andrianaplakoshmail@gmail.com}
\urladdr{https://www.imath.kiev.ua/~drozd}
\thanks{This work was supported within the framework of the project ``Innovative methods in the theory of differential equations, computational mathematics and mathematical modeling'', No.\,7/1/241 (Program of support of priority for the state scientific researches and scientific and technical (experimental)  developments of the Department of Mathematics NAS of Ukraine for 2022-2023).}
\keywords{Kleinian 4-group, regular lattices, cohomology, action of automorphisms, crystallographic groups, Chernikov groups}

\begin{abstract}
 We calculate explicitly cohomologies of the lattices over the Kleinian 4-group belonging to the regular components of the
 Auslander--Reiten quiver as well as of their dual modules. We also give a canonical form of cohomology classes under the action of
 automorphisms. The result is applied to the classification of some crystallographic and Chernikov groups.
\end{abstract}

\maketitle 

\tableofcontents


 The aim of this paper is to apply the results on cohomologies of the Kleinian 4-group \cite{dpl2} to the classification of crystallographic and Chernikov
 groups. For this purpose it is important to have an explicit presentation of 2-cocycles. We find such presentation for a special class of lattices,
 called \emph{regular}, and for their dual modules. Moreover, we describe the orbits of the action of automorphisms of modules on
 cohomologies. From this results we obtain a complete description of crystallographic and Chernikov groups with the Kleinian top and
 regular base.

 \section{Lattices over the Kleinian 4-group.}
 \label{s1}
 
   In what follows $K$ denotes the Kleinian 4-group, $K=\langle a,b \mid a^2=b^2=1,\,ab=ba \rangle$. We study cohomologies of this
  group with the values in $K$-\emph{lattices}, i.e. $K$-modules $M$ such that the additive group of $M$ is free abelian of finite rank,
  and in their \emph{duals}, i.e. the modules $\Hom_\mZ(M,\mQ/\mZ)$. Let $R=\mZ K$. We embed it into $\tR=\mZ^4$ identifying $a$ with
  the quadruple $(1,1,-1,-1)$ and $b$ with $(1,-1,1,-1)$. Note that $\tR$ is the integral closure of $R$ in $\mQ\*_\mZ R$.
   Let $\mZ_p$ be the ring of $p$-adic integers, $M_p=M\*_\mZ\mZ_p$ for every
  abelian group $M$. Then $R_p=\mZ_p^4$ for $p\ne2$ and $R_2\spe 4\mZ_2^4$. It follows from \cite[Th.\,3.7]{adeles} that two
  $K$-lattices $M,N$ are isomorphic \iff they are \emph{in the same genus}, i.e. $M_p\simeq N_p$ for all $p$. Moreover, if $p\ne2$, the 
  $R_p$-lattice $M_p$ is uniquely defined by the rational envelope $\mQ\*_\mZ M$. Therefore, a $K$-lattice $M$ is uniquely detremined
  by its $2$-adic completion, which we denote by $\hM$. We denote by $R\lat$ the category of $R$-lattices and by $\hR\lat$ 
  the category of \emph{$\hR$-lattices}, i.e. $\hR$-modules which are finitely generated and torsion free (hence free) as $\mZ_2$-modules.  
   The functor $M\mapsto\hM$ is a \emph{representation equivalence} between the categories $R\lat$ and $\hR\lat$,
  i.e. it maps non-isomorphic modules to non-isomorphic, indecomposable to indecomposable and every $\hR$-lattice is isomorphic to $\hM$
   for a uniquely defined $R$-lattice $M$. 
   
    Since $4\hH^n(K,M)=0$ for any $K$-module $M$ \cite[Prop.\,XII.2.5]{CE}, 
   $\hH^n(K,M)\simeq\hH^n(K,\hM)$. Let $\mD=\mQ_2/\mZ_2$, where $\mQ_2$ is the field of $p$-adic numbers. 
   It is the group \emph{of type $2^\8$}, i.e. the direct limit $\varinjlim_n{\mZ/2^n\mZ}$ of finite cyclic $2$-groups with respect to the
    natural embeddings $\mZ/2^n\mZ\to\mZ/2^{n+1}\mZ$. We call $K$-modules of the form $\Hom_\mZ(M,\mD)\simeq\Hom_{\mZ_2}(\hM,\mD)$, 
    where $M$ is a $K$-lattice, \emph{$K$-colattices}. 

    The ring $R$ is \emph{Gorenstein}, i.e. $\mathrm{inj.dim}_RR=1$. Since $R_p$ is a maximal order for $p\ne2$ 
    and $R_2$ is local, \cite[Lem.\,2.9]{qbass} implies that $R$ has a unique minimal proper overring $A$ and every indecomposable 
  $R$-lattice, except $R$ itself, is an $A$-lattice. Actually, $A$ coincides with the subring of $\tR=\mZ^4$ consisting 
  of all quadruples $(z_1,z_2,z_3,z_4)$ such that $z_1\equiv z_2\equiv z_3\equiv z_4\!\pmod2$. 
    
   Let $\gM$ be the ideal of $A$ consisting of  all quadruples $(z_1,z_2,z_3,z_4)$ such that 
   $z_1\equiv z_2\equiv z_3\equiv z_4\equiv0\!\pmod2$. Then $\hat{\gM}=\rad\hA=\rad\htR$.
   So $\hA$ is a \emph{Backstr\"om order} in the sense of \cite{rinrog}. Therefore, according to \cite{rinrog}, $\hA$-lattices,
   hence also $A$-lattices, are classified by the representations of the quiver 
  \[
    \La=  \vcenter
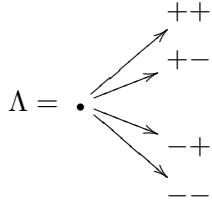
{\xymatrix@R=1ex{ & {++} \\ & {+-} \\ \bul \ar[uur] \ar[ur] \ar[dr] \ar[ddr] \\ &{-+} \\ &{--}	  }  }
  \]  
  Recall the corresponding construction (on the level of $A$-lattices). For any $K$-module $M$ let $M_{\al\be}$, where $\al,\be\in\{+,-\}$, be the submodule
  $\{u\in M\mid au=\al u,\,bu=\be u\}$. If $M$ is an $A$-lattice, set $M^\sharp=\tR M$ (the product of lattices inside $\mQ\*_\mZ M$, which coinsides with the
  image of the natural map $\tR\*_AM\to\mQ\*_\mZ M$). 
  It is the smallest $\tR$-module containing $M$. Then $M^\sharp=\bop_{\al\be}M^\sharp_{\al\be}$. Let $V_\bul=M/\gM M$ and
  $V_{\al\be}=M^\sharp_{\al\be}/2M^\sharp_{\al\be}$. Taking for $f_{\al\be}$  the natural maps $V_\bul\to V_{\al\be}$, 
  we obtain a representation of the quiver $\La$, which we denote by $\Phi(V)$. Thus we define a functor $\Phi$ from the category 
  $A\lat$ of $A$-lattices to the category $\rep\La$ of represenations of the quiver $\La$ over the field $\aK=\mZ/2\mZ$.  
  We have the following result analogous to that of \cite{rinrog}.
  
  \begin{theorem}\label{riro} 
   Let $\kR$ be the category of representations
  \begin{equation}\label{e01} 
     V=  \vcenter{\xymatrix@R=2ex@C=4em{ & V_{++} \\ & V_{+-} \\ 
    V_\bul \ar[uur]|{\,f_{++}\,} \ar[ur]|{\,f_{+-}\,} \ar[dr]|{\,f_{-+}\,} \ar[ddr]|{\,f_{--}\,} \\ & V_{-+} \\ & V_{--}	  }  }
 \end{equation}
 of the quiver $\La$ over $\aK$ such that all maps $f_{\al\be}$ are surjective and the induced map $f_\+:V_\bul\to V_\+=\bop_{\al\be}V_{\al\be}$ 
 is injective. The functor $\Phi$ is a representation equivalence $A\lat\to\kR$ such that all induced maps $\Phi_{M,N}:\Hom_A(M,N)\to\Hom_\La(\Phi(M),\Phi(N))$
 are surjective and 
 \begin{equation}\label{ker} 
 \ker\Phi_{M,N}=\Hom_A(M,\gM N)=2\Hom_{\tR}(\tM,\tN).
 \end{equation}
  \end{theorem} 
  \begin{proof}
  Obviously, always $\Phi(M)\in\kR$.  Let $V\in\kR$, $d_{\al\be}=\dim V_{\al\be}$ and $d_\bul=\dim V_\bul$.
   Denote by $\mZ_{\al\be}$ the $K$-module $\mZ$, where $a$ acts as $\al1$ 
 and $b$ acts as $\be1$. Thus $\tR=\bop_{\al\be}\mZ_{\al\be}$. 
 Set $\tM=\bop_{\al\be}\mZ_{\al\be}^{d_{\al\be}}$ and define $M(V)$ as the preimage of $\im f_\+$ in $\tM$ under the
 epimorphism $\tM\to\tM/2\tM\simeq V_\+$. Then $M(V)$ is an $A$-lattice such that $\Phi(M(V))\simeq V$
 and $M(V)^\sharp=\tM$. It is also evident that $M(\Phi(M))\simeq M$. Hence $\Phi$ is a representation equivalence.
 
   A morphism $\phi:V\to V'$, where
\[
     V'=  \vcenter{\xymatrix@R=2ex@C=4em{ & V'_{++} \\ & V'_{+-} \\ 
    V'_\bul \ar[uur]|{\,f'_{++}\,} \ar[ur]|{\,f'_{+-}\,} \ar[dr]|{\,f'_{-+}\,} \ar[ddr]|{\,f'_{--}\,} \\ & V'_{-+} \\ & V'_{--}	  }  }
\] 
is given by a quintuple of homomorphisms $\{\phi_\bul,\phi_{++},\phi_{+-},\phi_{-+},\phi_{--}\}$ such that 
$\phi_{\al\be}f_\bul=f'_{\al\be}\phi_\bul$ for all $\al,\be$. If $V=\Phi(M)$ and $V'=\Phi(N)$, these homomorphisms give a
homomorphism $\tilde{\phi}:\tM/2\tM\to\tN/2\tN$ such that $\tilde{\phi}f_\+=f'_\+\tilde{\phi}$. If we lift $\tilde{\phi}$ to a
homomorphism $\psi^\sharp:\tM\to\tN$, it implies that $\psi^\sharp(M)\sbe N$, so we obtain a homomorphism $\psi:M\to N$
such that $\Phi(\psi)=\phi$. Obviously, $\Phi(\psi)=0$ \iff $\im\psi\sbe\gM N=2\tN$.
  \end{proof}
    
   We call the quintuple $(d_\bul,d_{\al\be})\ (\al,\be\in\{+,-\})$ the \emph{dimension} of the representation $V$ or of the corresponding lattice $M=M(V)$, 
   denote it by $\Dim V$ or $\Dim M$ and usually present it in the form
 \[
   \rp{d_\bul}{d_{++}}{d_{+-}}{d_{-+}}{d_{--}}
 \]
 We also denote, if necessary, $d_\bul=d_\bul(M),\,d_{\al\be}=d_{\al\be}(M)$ and $d_\+=d_\+(M)=\sum_{\al\be}d_{\al\be}(M)$.
 Note that the rank of $M$ as of $\mZ$-module equals $d_\+(M)$, while $d_\bul=\dim_{\aK}M/\gM M$.
  
 We also need analogues of some results from \cite{rogback}. For this purpose we establish a lemma.
  For any $\mZ$-lattice or $\tR$-lattice $M$ we set $\oM=M/2M$ and if $\al:M\to N$ is a homomorphism of lattices, 
  we denote by $\oal$ the induced map $\oM\to\oN$. 

\begin{lemma}\label{lift} 
 Any exact sequence $0\to\oM\xarr{\oal}\oN\xarr{\obe}\oL\to0$ of $\mZ$-lattices or of $\tR$-lattices can be lifted to an exact sequence  
 $0\to M\xarr{\al} N\xarr{\be}L\to0$.
\end{lemma}
\begin{proof}
 As $\tR=\mZ^4$, we only have to consider the case of $\mZ$-lattices.
 We choose bases in $M,N,L$ and the corresponing bases in $\oM,\oN,\oL$ and identify $\oal$ and $\obe$ with their matrices \wrt these bases.
 There are invertible matrices $\oS,\oT$  of appropriate sizes over the field $\aK$ such that $\oS^{-1}\oal \oT=\smtr{I\\0}$, where $I$ is the unit matrix.
 Then $\obe\oS=\mtr{0&\oC}$ for some invertible matrix $\oC$. Since the maps $\GL(n,\mZ)\to\GL(n,\aK)$ are surjective, we can lift $\oS,\oT,\oC$ to
 invertible matrices $S,T,C$ over $\mZ$. Then the homomorphisms $M\to N$ and $N\to L$  respectively given by the matrices $\al=S\smtr{I\\0}T^{-1}$
  and $\be=\mtr{0& C}S^{-1}$ are the necessary liftings of $\oal$ and $\obe$.
\end{proof}

 \begin{corol} \label{replift}
  Let $0\to V'\xarr{\oal} V\xarr{\obe} V''\to 0$ be an exact sequence of representations from $\kR$. It can be lifted to an exact sequence
  of $A$-lattices $0\to M(V')\xarr{\al} M(V)\xarr{\be} M(V'')\to 0$.\!%
   \footnote{\,This result does not follow directly from \cite[Lem.\,4]{rogback}, where the case of algebras over comlete discrete
 valuation rings is considered. Moreover, it highly depends on Lemma~\ref{lift}, hence on the ``smallness'' of the residue field $\aK$.}
 \end{corol}
 \begin{proof}
 We denote $\tV=\bop_{\al\be}V_{\al\be}$. Then we have a commutative diagram
 \begin{equation}\label{ex1} 
  \vcenter{  \xymatrix{    0 \ar[r] & V_\bul' \ar[r]^{\tilde{\al}} \ar[d]_{f'_{V'}} &	  V_\bul \ar[r]^{\tilde{\be}}  \ar[d]_{f_{V}} &	 V_\bul'' \ar[r]  \ar[d]_{f''_{V''}} &	0 \\
   					 0 \ar[r] & \tV' \ar[r]^{\tilde{\al}} &	  \tV \ar[r]^{\tilde{\be}} &	 \tV'' \ar[r] &	0 } 
  }
 \end{equation}
 By Lemma~\ref{lift}, the second row can be lifted to an exact sequence 
 \[
  0\to M(V')^\sharp\xarr{\al^\sharp} M(V)^\sharp \xarr{\be^\sharp} M(V'')^\sharp\to 0. 
 \]
 Recall that $M(V)\spe 2M(V)^\sharp$ and $M(V)/2M(V)^\sharp\simeq V$.
 Since the diagram \eqref{ex1} is commutative, it implies that $\al^\sharp(M(V'))\sbe M(V)$ and $\be^\sharp(M(V))\sbe M(V'')$.
 So we obtain the necassary lifting 
 		$$0\to M(V')\xarr{\al} M(V)\xarr{\be} M(V'')\to 0.$$ 
 It is exact by the $3\xx3$ lemma,
 \end{proof}
  
    We say that a monomorphism of lattices $\phi:M\to N$ is \emph{pure} if $\cok\phi$ is torsion free (hence, also a lattice).
    
\begin{corol}\label{isolift} 
 Every epimorphism (monomorphism, isomorphism)  $V\to V'$ of representations from $\kR$ can be lifted to an epimorphism 
 (respectively, pure monomorphism, isomorphism) of $A$-lattices $\Phi(V)\to\Phi(V')$.
 \end{corol}
 
\begin{corol}\label{sublift} 
  Given a chain of subrepresentations 
  \[
     V=V_0\sp V_1\sp V_2\sp \dots\sp V_{m-1}\sp V_m=0,
  \]
  there is a chain of sublattices in $M=M(V)$
    \[
     M=M_0\sp M_1\sp M_2\sp \dots\sp M_{m-1}\sp M_m=0
  \]
  such that $M_k\simeq M(V_k)$ and $M_k/M_{k+l}\simeq M(V_k/V_{k+l})$ for all possible values of $k,l$.
 \end{corol}  
 
 \section{Regular lattices}
 \label{s2}
 
 Recall the structure of the \AR\ quiver $\kQ$ of the category $\hA\lat$ \cite{plak1}. According to \cite{rogback}, it is obtained from
 the \AR\ quiver of the category $\rep\La$ by gluing the preprojective and the preinjective components into one component.
 The resulting \emph{preprojective-preinjective component} is shown at Figure~1. 
 Here all lattices are uniquely determined by their dimensions. 
 The \AR\ transpose $\tau$ of the category $\hR\lat$ acts on this component as the shift to the left. Note that, 
 for every $A$-lattice $M$, $\tau\hM\simeq \Om \hM$, the syzygy of $\hM$ as of $\hR$-module  \cite[Prop.\,1.1]{dpl2}. Hence,
 $\Om M\simeq N$ if $\tau\hM=\hat N$. 
 
    \begin{sidewaysfigure}[p] 
  \vspace*{4.5in}
   {\footnotesize
    $$
\xymatrix@C=1em@R=1.5em{
& \txt{$\rp41222$} \ar[ddr] && \txt{$\rp32111$} \ar[ddr] && \txt{$\rp20111$}\ar[ddr] && \txt{$\rp11000$} \ar[ddr]&& 
\txt{$\rp10111$} \ar[ddr] && \txt{$\rp22111$}\ar[ddr] && \txt{$\rp31222$} \ar[ddr]&& \\
&  \txt{$\rp42122$} \ar[dr] && \txt{$\rp31211$} \ar[dr] && \txt{$\rp21011$}\ar[dr] && \txt{$\rp10100$} \ar[dr] &&
  \txt{$\rp11011$} \ar[dr] && \txt{$\rp21211$}\ar[dr] && \txt{$\rp32122$} \ar[dr] &&\\
 \cdots\ar[uur]\ar[ur]\ar[dr]\ar[ddr] &&  \txt{$\rp73333$} \ar[uur]\ar[ur]\ar[dr]\ar[ddr] && \txt{$\rp52222$} \ar[uur]\ar[ur]\ar[dr]\ar[ddr] 
 && \txt{$\rp31111$} \ar[uur]\ar[ur]\ar[dr]\ar[ddr] 
  && \txt{$\rp11111$} \ar[uur]\ar[ur]\ar[dr]\ar[ddr] && \txt{$\rp32222$} \ar[uur]\ar[ur]\ar[dr]\ar[ddr] && \txt{$\rp53333$} \ar[uur]\ar[ur]\ar[dr]\ar[ddr] 
  && \txt{$\rp74444$} \ar[uur]\ar[ur]\ar[dr]\ar[ddr] &\cdots \\
 & \txt{$\rp42212$} \ar[ur] && \txt{$\rp31121$} \ar[ur] && \txt{$\rp21101$}\ar[ur] && \txt{$\rp10010$} \ar[ur] && 
 \txt{$\rp11101$} \ar[ur] && \txt{$\rp21121$}\ar[ur] && \txt{$\rp32212$} \ar[ur] && \\
  &\txt{$\rp42221$} \ar[uur] && \txt{$\rp31112$} \ar[uur] && \txt{$\rp21110$} \ar[uur] && \txt{$\rp10001$} \ar[uur] && 
  \txt{$\rp11110$} \ar[uur] && \txt{$\rp21112$} \ar[uur] && \txt{$\rp32221$} \ar[uur] && } 
 $$}
 \vskip1cm
   \caption{Preprojective-preinjective component}
   \end{sidewaysfigure}
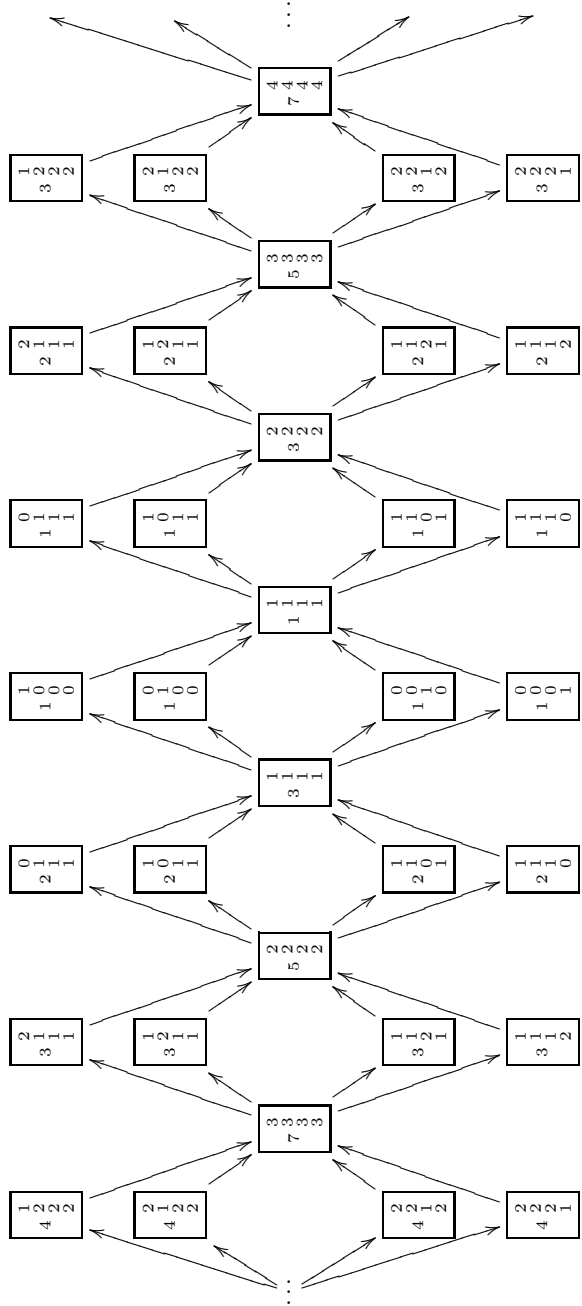 
   
 The other components, called \emph{regular},  are \emph{tubes}. They are parametrized by the closed points of the projective line
 $\mP^1_\aK$ (considered as a scheme), which consist of monic irreducible polynomials from $\aK[t]$ and the symbol $\8$. 
 We denote the tube corresponding to the polynomial $f(t)$ by $\kT^f$ and
  the tube corresponding to $\8$ by $\kT^\8$. We also write $\kT^0$ instead of $\kT^t$ and $\kT^1$ instead of $\kT^{t-1}$. When describing tubes,
  we substitute $A$-lattices for their completions and say that $T$ belongs to a tube $\kT^f$ if $\hat{T}$ belongs to this tube. Then we call $T$ a
  \emph{regular $K$-lattice}. We call a $K$-lattice $M$ \emph{regular} if all its indecomposable direct summands are regular.
  
  All tubes except $\kT^\la$ ($\la\in\{0,1,\8\}$) are \emph{homogeneous}, i.e. $\tau M\simeq M$ for all $M$ from this tube. They are of the form
  \begin{equation}\label{htube} 
      \xymatrix{ T^f_1 \ar@/^/[r]  & T^f_2    \ar@/^/[r] \ar@/^/[l] & T^f_3  \ar@/^/[r] \ar@/^/[l] & \cdots \ar@/^/[l]} 
  \end{equation}
   $\Dim T^f_m=\rp{2dm}{dm}{dm}{dm}{dm}$, where $d=\deg f(t)$. Actually, $T^f_m=M(\oT^f_m)$, where $\oT^f_m$ is the following representation 
   of the quiver $\La$:
  \vspace*{-2ex}
  \begin{equation}\label{tuberep} 
  	 \xymatrix@R=1em{ && {\aK^{dm}} \\ && {\aK^{dm}} \\ {\aK^{2dm}} \ar[uurr]|{\smtr{I&0}} \ar[urr]|{\smtr{0&I}} \ar[drr]|{\smtr{I&I}} \ar[ddrr]|{\smtr{I&F}}
	\\  && {\aK^{dm}} \\ && {\aK^{dm}} 	  }
  \end{equation}
  Here $I$ is the $dm\xx dm$ unit matrix and $F$ is the Frobenius matrix with the characteristic polynomial $f(t)^m$. 
  
     The tube $\kT^\la$ for $\la\in\{0,1,\8\}$ is of the form
 \begin{equation}\label{stube} 
  \vcenter{ \xymatrix{ T^{\la1}_1 \ar[r]  & T^{\la1}_2 \ar[r] \ar[dl] & T^{\la1}_3 \ar[r] \ar[dl]& T^{\la1}_4 \ar[r] \ar[dl] \ar[r] \ar[dl] & \cdots \ar[dl]\\
   T^{\la2}_1 \ar[r]  & T^{\la2}_2 \ar[r] \ar[ul] & T^{\la2}_3 \ar[r] \ar[ul]& T^{\la2}_4 \ar[r] \ar[ul] \ar[r] \ar[ul] & \cdots \ar[ul]} }
 \end{equation}
  Here  $\tau T^{\la1}_n\simeq T^{\la2}_n$ and $\tau T^{\la2}_n\simeq T^{\la1}_n$. For $\la=1$ we have
  \begin{align*}
  & \Dim T^{1j}_{2m}=\rp{2m}mmmm \ \text{ for both $j=1$ and $j=2$},\\
   & \Dim T^{11}_{2m-1}=\rp{2m-1}mm{m-1}{m-1},\\
   & \Dim T^{12}_{2m-1}=\rp{2m-1}{m-1}{m-1}mm.   
  \end{align*}
  Actually, $T^{11}_{2m}=M(\oT^{11}_{2m})$, where $\oT^{11}_{2m}$ is of the form \eqref{tuberep} with $d=1$ and $F=J_1$ being the 
   Jordan $m\xx m$ matrix with eigenvalue $1$. 
  $T^{11}_{2m-1}=M(\oT^{11}_{2m-1})$, where $\oT^{11}_{2m-1}$ is of the form
  
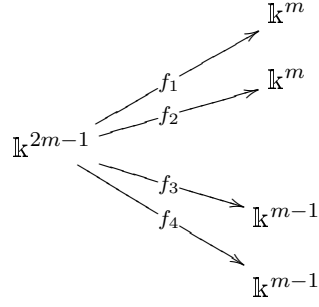
\[
 	 \xymatrix@R=1em{ && {\aK^{m}} \\ && {\aK^{m}} \\ {\aK^{2m-1}} \ar[uurr]|{f_1} \ar[urr]|{f_2} \ar[drr]|{f_3} \ar[ddrr]|{f_4}
	\\  && {\aK^{m-1}} \\ && {\aK^{m-1}} 	  }   
  \]
  Here
  \begin{align*}
   f_1&=\mtr{I&0&0\\0&0&1},\\
   f_2&=\mtr{0&I&0\\0&0&1},\\
   f_3&=\mtr{I&I&0},\\
   f_4&=\mtr{I&J_1&\mathrm{e}},\ \text{with } \mathrm{e}=\smtr{0\\\vdots\\0\\1},
  \end{align*}
  where the matrices $I$ and $J_1$ are of size $(m-1)\xx(m-1)$
  The lattices $T^{12}_n$ are obtained from $T^{11}_n$ by the permutations, respectively, of $f_1$ with $f_3$ and of $f_2$ with $f_4$.
  
  The tube $\kT^0$ ($\kT^\8$) is obtained from the tube $\kT^1$ by the permutation of $f_2$ with $f_4$ (respectively, of $f_2$ with $f_3$).
  
   Note that an indecomposable lattice $M$ belongs to a tube \iff  
  \begin{align}  
  &\textstyle   2d_\bul(M)=\sum_{\al\be}d_{\al\be}(M). \label{dtube} \\[-1ex]
  \intertext{  In this case }
  &  d_\bul(\Om M)=d_\bul(M)\ \text{ and }\ d_{\al\be}(\Om M)=d_\bul(M)-d_{\al\be}(M).  \label{dtube1} 
    \end{align}     
   
    The structure of the representations of quivers belonging to tubes is described in \cite{DR,DF} (see also \cite[Thm.\,31\,and\,36]{quivers}). Together with
  Corollary~\ref{sublift} it gives the following result for lattices. From now on, writing $T^f_m$, we always suppose that $f\notin\{t,t-1\}$, and writing $T^{\la j}_m$,
  we always suppose that $\la\in\{0,1,\8\}$ and $\la\in\{1,2\}$.
  
 \begin{theorem}\label{ptubes} 
 Every module $T^f_m$ or $T^{\la j}_m$  has a chain of submodules 
  \begin{equation}\label{chain} 
   M=M_0\sp M_1\sp M_2\sp\dots\sp M_{m-1}\sp M_m=\{0\}
  \end{equation} 
  such that
  \begin{enumerate}
  \item 
 $
   M_k/M_{k+l}\simeq
   \begin{cases}
    T^f_l &\text{if } M=T^f_m,\\
    T^{\la j}_l &\text{if } M=T^{\la j}_m \text{ and $k$ is even},\\
    T^{\la i}_l, &\text{where $i\ne j$, if } M=T^{\la j}_m \text{ and $k$ is odd.}    
   \end{cases} 
 $

 \smallskip 
 \item  The maps $T^f_m\to T^f_{m+1}$ and $T^{\la j}_m\to T^{\la j}_{m+1}$ in the diagrams, respectively, \eqref{htube} and \eqref{stube} 
 can be chosen injective, with the quotients, respectively, $T^f_1$  and $T^{\la j}_1$.
   \smallskip
  \item  The maps $T^f_{m+1}\to T^f_m$ and $T^{\la i}_{m+1}\to T^{\la j}_m\ (i\ne j)$  in the diagrams, respectively, \eqref{htube} and 
  \eqref{stube} can be chosen surjective, with the kernels, respectively, $T^f_1$ and $T^{\la i}_1$. 
  \smallskip
  \item  If $M$ and $M'$ belong to different tubes, then $\im\vi\sbe2N$ for every homomorphism $\vi:M\to N$.
   \end{enumerate}
 \end{theorem}
 
   One can aslo get a description of endomorphisms of indecomposable lattices belonging to tubes.

  For any commutative ring $K$ and any monic polynomial $f(t)$ of degree $d$ there is a natural embedding 
  $\eps^f_m:K[t]/(f(t)^m)\to\Mat(dm,K)$ arising from the regular representation of this quotient algebra over $K$. 
  We denote by $K^f_m$ the image of the embedding $K[t]/(f(t)^m)\to\Mat(dm,K)^4$
  such that all its components are $\eps^f_m$. We also denote by $K^t_{m*}$ the image of the embedding
  $K[t]/(t^m)\to\Mat(m,K)^2\xx\Mat(m-1,K)^2$ such that the first two components are $\eps^t_m$ and the other two
  components are the compositions $\eps^t_{m-1}\pi_m$, where $\pi_m$ is the surjection $K[t]/(t^m)\to K[t]/(t^{m-1})$.
  Certainly, if $m=1$, it is just the diagonal embedding $K\to K\xx K$.
  
 If $M$ is an $A$-lattice and $\Dim M=\rp{d_\bul}{d_{++}}{d_{+-}}{d_{-+}}{d_{--}}$,  the endomorphism ring $\End_AM$
 naturally embeds into $\End_{\tR}\tM=\prod_{\al\be}\Mat(d_{\al\be},\mZ)$ and we identify it with the image of this embedding.
 
  \begin{theorem}\label{etubes} 
   For an irreducible monic polynomial $f(t)\in\aK[t]$ choose a monic polynomial $\Tf(t)\in\mZ[t]$ such that 
   $f(t)=\Tf(t)\!\pmod2$.
   \begin{enumerate}
   \item 
   $\End_AT^f_m=\mZ^{\Tf}_m+2\Mat(dm,\mZ)^4,$
   where $d=\deg f(t).$
   
   \item 
  $\End_AT^{\la j}_{2m}=\mZ^t_m+2\Mat(m,\mZ)^4$  and \\
    $\End_AT^{\la j}_{2m-1}=\mZ^t_{m*}+2\big(\Mat(m,K)^2\xx\Mat(m-1,K)^2\big).$
   \end{enumerate}
  \emph{Obviously, it does not depend on the choice of $\Tf(t)$.} 
  \end{theorem}
  \begin{proof}
   It is well known (and can be verified by straightforward calculations) that $\End_\La \oT^f_m=\aK^f_m$, 
   while then $\End_\La \oT^{\la j}_{2m}=\aK^t_m$ and $\End_\La \oT^{\la j}_{2m-1}=\aK^t_{m*}.$
   It remains to apply Theorem~\ref{riro}.
  \end{proof}

  \section{Cohomologies}
 \label{s3} 
 
   We will give an explicit description of the
  cohomologies $H^n(K,M)$ and $H^n(K,DM)$ for $n>0$ and regular $K$-lattices $M$. Obviously, we only have to calculate them
  for indecomposable lattices.
  
  It follows from \cite{dpl1} that a free resolution $\bP$ for the trivial $K$-module $\mZ$ can be chosen as follows: $\bP_n$ is the set of homogeneous
  polynomials of degree $n$ from $R[x,y]$ and
  \[
   d(x^ky^l)=(a+(-1)^k)x^{k-1}y^l +(-1)^k(b+(-1)^l)x^ky^{l-1}.  
  \]
  So an $n$-cocycle $\ga$ is given by the values $\ga(x^my^{n-m})\ (0\le m\le n)$.
 
    Let $M$ be an indecomposable regular lattice. Set
    \label{ksi}
 \[
 M(n)=
 \begin{cases}
  M_{++} &\text{if $n$ is even},\\
  M_{-+}  &\text{if $n$ is odd and } M\notin \kT^\8,\\
  M_{+-}  &\text{if $n$ is odd and } M\in \kT^\8,\\  
 \end{cases}
 \]
 We define a homomorphism $\xi:M(n)\to H^n(K,M)\ (n>0)$. It sends an element $v\in M(n)$
 to the class of the cocycle $\xi_v$ which is defined as follows.
 
\begin{itemize}
\item   If $n$ is even, $M\notin \kT^\8$ and $v\in M_{++}$,
 $$ \xi_v(x^my^{n-m})=
 \begin{cases}
  v & \text{ if } m=n,\\
  0 & \text{ otherwise}.
 \end{cases}$$ 
 
\item   If $n$ is even, $M\in \kT^\8$ and $v\in M_{++}$,
 $$ \xi_v(x^my^{n-m})=
 \begin{cases}
  v & \text{ if } m=0,\\
  0 & \text{ otherwise}.
 \end{cases}$$ 

\item   If $n$ is odd, $M\notin \kT^\8$ and $v\in M_{-+}$,
 $$  \xi_v(x^my^{n-m})=
 \begin{cases}
  v & \text{ if } m=n,\\
  0 & \text{ otherwise}.
 \end{cases}$$ 
 
\item   If $n$ is odd, $M\in \kT^\8$ and $v\in M_{+-}$,
 $$  \xi_v(x^my^{n-m})=
 \begin{cases}
  v & \text{ if } m=0, \\
  0 & \text{ otherwise}.
 \end{cases}$$  
\end{itemize}
One easily verifies that $\xi_v$ is indeed a cocycle.
 
 \begin{theorem}\label{HnZ} 
  For every indecomposable regular $K$-lattice $M$ and every $n>0$ the map $\xi$ induces an isomorphism $\oxi:M(n)/2M(n)\ito H^n(K,M)$.
 \end{theorem}
 
 For the proof we use the following results.
 
\begin{lemma}\label{l1} 
 If $0\to M'\to M\to M''\to 0$ is an exact sequence of regular $K$-lattices, then the induced sequences
 \begin{align}
  & 0\to \Om M'\to \Om M\to \Om M''\to 0, \label{om} \\\ 
   &0\to \Om^{-1} M'\to \Om^{-1} M\to \Om^{-1} M''\to 0 \label{iom} 
 \end{align}
 are also exact.
\end{lemma}
\begin{proof}
 From the properties of syzygies it follows that there is an exact sequence 
 $$0\to\Om M'\to\Om M\+P\to\Om M''\to 0$$ for some projective $R$-module $P$. 
 But, as $\Om M=\tau M$, the formulae \eqref{dtube} and \eqref{dtube1} show that $d_{\al\be}(\Om M)=d_{\al\be}(\Om M')+d_{\al\be}(\Om M'')$.
 Therefore, $P=0$ and we get the exact sequence \eqref{om}. Then \eqref{iom} follows by duality.
\end{proof}

  \begin{corol}\label{l2} 
  If $0\to M'\to M\to M''\to 0$ is an exact sequence of regular $K$-lattices, the induced sequence of cohomologies
  \[
  0\to \hH^n(K,M')\to \hH^n(K,M)\to \hH^n(K,M'')\to 0
  \] 
 is also exact for every $n\in\mZ$.
 \end{corol}
 \begin{proof}
  It is known \cite[Lem.\,2.2]{dpl2} that if $M$ contains no direct summands isomorphic to  $\mZ_{++}$, 
  then $\hH^0(K,M)\simeq\aK^{d_{++}(M)}$. It implies
  the claim for $n=0$. The general case follows from Lemma~\ref{l1} and the known fact that 
  $\hH^n(K,M)\simeq\hH^{n+1}(K,\Om M)\simeq\hH^{n-1}(K,\Om^{-1}M)$.
 \end{proof}

 \begin{proof}[Proof of Theorem~\ref{HnZ}]
  Note that \cite[Th.~2.3]{dpl2} shows that $M(n)/2M(n)\simeq H^n(K,M)$. Hence we only have to check that $\oxi$ is injective.
  First, we check the claim for the lattices $T^f_1$ and $T^{\la j}_1$.  As the calculations are quite similar, we only consider the case of 
  $M=T^{11}_1$ and $n$ even (it seems the most complicated). Then $M$ is the submodule of $\mZ_{++}\+\mZ_{--}$ consisting of the pairs 
  $(z,z')$ such that $z\equiv z'\!\pmod2$. Thus  the basic element of $M(n)$ is $v=(2,0)$. We have to check that $\xi_v\ne\dd\ga$ 
  for any map $\ga:\bP_{n-1}\to M$. Suppose that $\xi_v=\dd\ga$. Note that if $\ga(x^{n-1})=(z,z')$, then $\dd\ga(x^n)=(2z,0)$, 
  whence $z=1$. Let $\ga(x^{n-k-1}y^k)=(z_k,z_k')\ (0< k<n)$. Then
  \[
  \dd\ga(x^{n-1}y)=(a-1)(z_1,z'_1)-(b-1)(z,z')=(0,-2z'+2z_1') =(0,0), 
  \]
  hence $z_1'=z'\equiv 1\!\pmod2$;
  \[
    \dd\ga(x^{n-2}y^2)=(a+1)(z_2,z'_2)+(b+1)(z_1,z_1')=(2z_2+2z_1,0) =(0,0), 
  \]
  hence $z_2=-z_1\equiv 1\!\pmod2$. Repeating this process, we obtain that all $z_k\equiv 1\!\pmod2$, so $z_k\ne0$. Then 
  $\dd\ga(y^n)=(2z_{n-1},0)\ne0=\xi_v(y^n)$ and we get a contradiction.
  
  For the lattice $M=T^f_m$ or $T^{\la j}_m \ (m>1)$ we have an exact sequence $0\to M'\to M\to M''\to 0$, where $M'\simeq T^f_1$ or $T^{\la j}_1$
  and $M''\simeq T^f_{m-1}$ or $T^{\la i}_{m-1}$. It gives a commutative diagram with exact rows
  \[
 \xymatrix{ 0 \ar[r] & M'(n) \ar[r]\ar[d]^{\xi} &	M(n) \ar[r]\ar[d]^{\xi} & M''(n) \ar[r]\ar[d]^{\xi} & 0  \\
 				 0 \ar[r] & H^n(K,M') \ar[r] &H^n(K,M) \ar[r] & H^n(K,M'') \ar[r] & 0 	}  
  \]
  Using induction, we can suppose that the first and the third vertical maps satisfy the assertion of the theorem. Then the same is true for
  the second vertical map.
 \end{proof} 
 
 Analogous considerations give an explicit description of the cohomologies for \emph{regular colattices}, i.e. the dual modules of regular lattices. 
 For an indecomposable regular colattice $N=DM$ set $\bar{N}=\setsuch{u\in N}{2u=0}$ and
  \[
 N(n)=
 \begin{cases}
  \bar{N}_{++} &\text{if $n$ is odd}\\
  \bar{N}_{-+}  &\text{if $n$ is even and } M\notin T^\8,\\
  \bar{N}_{+-}  &\text{if $n$ is even and } M\in T^\8.
 \end{cases}
 \]
  We define a homomorphism $\eta:N(n)\to H^n(K,N)\ (n>0)$.  It sends an element $u\in N(n)$
 to the class of the cocycle $\eta_u$ which is defined as follows.
 
 \label{eta} 
 \begin{itemize}
 \item   If $n$ is even, $M\notin \kT^\8$ and $u\in \bar{N}_{-+}$,
 $$ \eta_u(x^my^{n-m})=
 \begin{cases}
  u & \text{ if } m=n,\\
  0 & \text{ otherwise}.
 \end{cases}$$ 
 
  \item  If $n$ is even, $M\in \kT^\8$ and $u\in \bar{N}_{+-}$,
 $$ \eta_u(x^my^{n-m})=
 \begin{cases}
  u & \text{ if } m=0,\\
  0 & \text{ otherwise}.
 \end{cases}$$ 
 
 \item  If $n$ is odd, $u\in \bar{N}_{++}$,
 $$  \eta_u(x^my^{n-m})=
 \begin{cases}
  u & \text{ if } m=n,\\
  0 & \text{ otherwise}.
 \end{cases}$$

 \end{itemize}
 One easily verifies that $\eta_u$ is indeed a cocycle.
 
 \begin{theorem}\label{HnD} 
  For every indecomposable regular colattice $N$ and every $n>0$ the map $\eta$ induces an isomorphism $N(n)\ito H^n(K,N)$.
 \end{theorem}
 \noindent
 We omit the proof since it is quite analogous to that of Theorem~\ref{HnZ} (and even easier).

  \label{s4} 
 \section{Action of automorphisms}
 
 \subsection{Lattices}
 \label{s31} 
 
 We also need to know how automorphisms of lattices and of the group act on cohomologies. Let $M=T^f_m$ or $M=T^{\la j}_m$. Consider the chain of
 submodules $M_k\sb M$ from Theorem~\ref{ptubes}. We denote by $E^f_{k,m,n}$ or, respectively, by $E^{\la j}_{k,m,n}$,
 where $0\le k<m$, the set $M_k(n)\setminus \big(2M_k(n)+ M_{k+1}(n)\big)$. Note that $E^{\la j}_{k,m,n}\ne\0$ \iff $k<m$ and
 \begin{equation}\label{km} 
 k\equiv
  \begin{cases}
  j &\text{if $n$ is even}\\ 
  j+1 &\text{if $n$ is odd}
  \end{cases}\!\pmod2
 \end{equation}
 \\ 
  Theorems~\ref{ptubes} and \ref{etubes} easily imply the following result.
  
  \begin{theorem}\label{act} 
  Let $e\in E^f_{k,m,n}$ or $e\in E^{\la j}_{k,m,n}$, and $e'\in E^f_{k',m',n}$ or, respectively, $e'\in E^{\la j'}_{k',m',n}$. There is a homomorphism
  $\theta:T^f_m\to T^f_{m'}$ or, respectively, $\theta:T^{\la j}_m\to T^{\la j'}_{m'}$ such that $\theta(e)=e'$ \iff either $m\ge m'$ and $k\le k'$ or
  $m\le m'$ and $k\le k'-m'+m$. If $m=m'$ and $k=k'$, $\theta$ can be chosen as an isomorphism.
  \end{theorem}
 
 \begin{defin}\label{non} 
 \begin{enumerate}
 \item We fix for every quadruple $(f,m,k,n)$, where $k<m$, an element $e^f_{m,k,n}\in E^f_{m,k,n}$ and for every quintuple $(\la,j,m,k,n)$, where
  $k<m$ and $k,j$ satisfy the condition \eqref{km}, an element $e^{\la j}_{m,k,n}\in E^{\la j}_{m,k,n}$.
 
 \medskip
 \item   For a homogeneous tube $\kT^f$ we call a \emph{standard sequence} a sequence $\si=(m_i,k_i)\ (1\le i\le s)$, where 
 $m_1>m_2>\dots>m_s$, $1\le k_i< m_i$ and
 $k_{i'}<k_i<k_{i'}+m_i-m_{i'}$ for $ i<i'$. We set $M^f_\si=\bop_{i=1}^sT^f_{m_i}$ and 
 $e^f_{\si,n}=\sum_{i=1}^se^f_{m_i,k_i,n}$.

\medskip
 \item  For a special tube $\kT^\la$ we call a \emph{standard sequence} a sequence $\si=(j_i,m_i,k_i)\ (1\le i\le s)$, where $j_i\in\{1,2\}$,
 $m_1>m_2>\dots>m_s$, $1\le k_i< m_i$ and  $k_{i'}<k_i<k_{i'}+m_i-m_{i'}$ for $ i<i'$. We set $M^\la_\si=\bop_{i=1}^sT^{\la j_i}_{m_i}$.
 We call such sequence
 \begin{itemize}
  \item  \emph{even} if $k_i\equiv j_i\!\pmod2$ for all $i$,
  \item    \emph{odd} if $k_i\equiv j_i+1\!\pmod2$ for all $i$.
  \end{itemize} 
  For an even (odd) standard sequence and even (respectively, odd) $n$ we set  $e^\la_{\si,n}=\sum_{i=1}^s e^{\la j_i}_{m_i,k_i,n}$.
    
    \medskip
    \item  
    We define \emph{standard data} as a pair $\De=(\Si,S)$, where $\Si=\{\kT^{f_q}\}$  $1\le q\le r$ is a set of different tubes
    and $S=\{\si_q\}$ $(1\le q\le r)$ is a set of standard sequences $\si_q$ for each tube $T^{f_q}$.  We call such data \emph{special}
    if at least one of the tubes $\kT^{f_q}$ is special. Special standard data are said to be even  or odd if all standard sequences $\si_q$
    for special tubes $\kT^{f_q}$ are so. We set $M_\De=\bop_{q=1}^rM^{f_q}_{\si_q}$ and $e_{\De,n}=\sum_{q=1}^s e^{f_q}_{\si_q,n}$.
   In the latter definition we suppose that, if $\De$ is special, it is even if $n$ is even and it is odd if $n$ is odd (otherwise the element
    $e_{\De,n}$ is not defined).
 
    \vskip0.5em\noindent
    If the tube $\kT^\8$ occurs in $\Si$, say, $f_k=\8$, 
  we denote by $e_{\De,n}^\8=e^\8_{\si_k,n}$ and by $e_{\De,n}^0$ the rest of the sum defining $e_{\De,n}$. 
  Of course, it is possible that $e_{\De,n}^\8=0$ or $e_{\De,n}^0=0$.
\end{enumerate}
 \end{defin}
 
 Theorem~\ref{act} implies the following result.
 
 \begin{theorem}\label{cohom} 
  Let $M$ be a regular $R$-lattice and $\eps\in H^n(K,M)\ (n>0)$. 
  There are standard data $\De$ and a decomposition $\theta:M\ito M_0\+M_\De$
   such that the projection of $\theta(\eps)$ onto $H^n(K,M_0)$ is zero and 
  the projection of  $\theta(\eps)$ onto $H^n(K,M_\De)$ equals $\oxi(e_{\De,n})$
  \emph{(see page~\pageref{ksi} for the definition of $\oxi$).}
  
  \emph{If $\eps=0$, $M_\De=0$.}
 \end{theorem}
  
  In particular, we obtain a description of orbits of automorphisms of indecomposable regular lattices on cohomologies.
    
 \begin{corol}\label{orbit} 
  Let $M$ be an indecomposable regular lattice. Consider the chain \eqref{chain} of its submodules and denote by
  $H^n_k(K,M)$ the image in $H^n(K,M)$ of $H^n(K,M_k)$. Then the orbits of $\aut_KM$ on $H^n(K,M)\ (n>0)$
  are $H^n_k(K,M)\setminus H^n_{k+1}(K,M)\ (0\le k< m)$ and $\{0\}$. 
 \end{corol}
 
 The group of automorphisms of the group $K$ is the symmetric group $S_3$: it just permutes the elements $a,b$ and $c=ab$. Its generators are the
 transpositions $\tau_2:a\leftrightarrow b$ and $\tau_3:a\leftrightarrow c$. They permute the $+-$ component of the diagram \eqref{e01}, 
 respectively, with the $-+$ component and with the $--$ component. Thus $\tau_2$ permutes $\kT^1$ and $\kT^0$,
 while $\tau_3$ permutes $\kT^1$ and $\kT^\8$. Rather simple matrix calculations show that $\tau_2$ permutes $\kT^f$ with $\kT^{f^{(2)}}$,
 while $\tau_3$ permutes $\kT^f$ with $\kT^{f^{(3)}}$, where
 \begin{align*}
  & f^{(2)}(t)=f(1)^{-1}(t-1)^df\left(\dfrac{t}{t-1}\right),\\
  & f^{(3)}(t)=(-1)^df(1-t),
 \end{align*}
 where $d=\deg f$. 
 It induces the action of $S_3$ on the set of standard data. 
 Note that, if $\psi\in\aut K$, there is an automorphism $\vi\in\aut M_{\tau\De}$ such that $\psi\xi(e_{\De,n})=\vi\xi(e_{\psi\De,n})$.
 
\subsection{Colattices}
\label{s32}

 Let now $N=DM$, where $M$ is a $K$-lattice. If $M$ is regular, we call $N$ regular too.
  If $N=DM$, where $M=T^f_m$ or $M=T^{\la j}_m$, there is a chain of
 submodules, dual to the chain \eqref{chain} from Theorem~\ref{ptubes}
 \begin{equation}\label{dchain} 
   0=N_0\sb N_1\sb N_2\sb \dots\sb N_{m-1} \sb N_m=N,
 \end{equation}
  where $N_k=M_k^\perp$ and
 \[
  N_{k+l}/N_k\simeq 
   \begin{cases}
    DT^f_l &\text{if } N=DT^f_m,\\
    DT^{\la j}_l &\text{if } N=DT^{\la j}_m \text{ and $k$ is even},\\
    DT^{\la i}_l, &\text{where $i\ne j$, if } N=DT^{\la j}_m \text{ and $k$ is odd.}   
  \end{cases} 
 \]
 We denote by $Z^f_{k,m,n}$ or, respectively, by $Z^{\la j}_{k,m,n}$ the set $N_{k+1}(n)\setminus N_k$.
 Again, $Z^{\la j}_{k,m,n}\ne\0$ \iff $k<m$ and the condition \eqref{km} holds.

The duality gives analogues of Theorems~\ref{act} and  \ref{cohom}.
  
  \begin{theorem}\label{dact} 
    Let $z\in Z^f_{k,m,n}$ or $z\in Z^{\la j}_{k,m,n}$, and $z'\in Z^f_{k',m',n}$ or, respectively, $z'\in Z^{\la j'}_{k',m',n}$. There is a homomorphism
  $\theta: DT^f_m\to DT^f_{m'}$ or, respectively, $\theta: DT^{\la j}_m\to DT^{\la j'}_{m'}$ such that $\theta(z)=z'$ \iff either $m\le m'$ and $k\ge k'$ or
  $m\ge m'$ and $k\ge k'-m'+m$. If $m=m'$ and $k=k'$, $\theta$ can be chosen as an isomorphism.
  \end{theorem}
  
   \begin{defin}\label{dnon} 
     \begin{enumerate}
 \item We fix for every quadruple $(f,m,k,n)$, where $k<m$, an element $z^f_{m,k,n}\in Z^f_{m,k,n}$ and for every quintuple $(\la,j,m,k,n)$, where
  $k<m$ and $k,j$ satisfy the condition \eqref{km}, an element $z^{\la j}_{m,k,n}\in Z^{\la j}_{m,k,n}$.
 
 \medskip
 \item   For a homogeneous tube $\kT^f$ we call a \emph{costandard sequence} a sequence $\si=(m_i,k_i)\ (1\le i\le s)$, where 
 $m_1<m_2<\dots<m_s$, $1\le k_i< m_i$ and
 $k_{i'}>k_i>k_{i'}+m_i-m_{i'}$ for $ i<i'$. We set $N^f_\si=\bop_{i=1}^sP^f_{m_i}$ and 
 $z^f_{\si,n}=\sum_{i=1}^se^f_{m_i,k_i,n}$.

\medskip
 \item  For a special tube $\kT^\la$ we call a \emph{costandard sequnce} a sequence $\si=(j_i,m_i,k_i)\ (1\le i\le s)$, where $j_i\in\{1,2\}$,
 $m_1<m_2<\dots<m_s$, $1\le k_i< m_i$, $1\le k_i\le m_i$ and  $k_i<k_{i'}<k_i+m_{i'}-m_i$ for $i'<i$. We set 
 $N^\la_\si=\bop_{i=1}^sT^{\la j_i}_{m_i}$.
 We call such sequence
 \begin{itemize}
  \item  \emph{even} if $k_i\equiv j_i\!\pmod2$ for all $i$,
  \item    \emph{odd} if $k_i\equiv j_i+1\!\pmod2$ for all $i$.
  \end{itemize} 
  For an even (odd) costandard sequence and even (respectively, odd) $n$ we set  
    $z^\la_{\si,n}=\sum_{i=1}^s z^{\la j_i}_{m_i,k_i,n}$.
    
       \medskip
    \item We define \emph{costandard data} as a pair $\De=(\Si,S)$, where $\Si=\{\kT^{f_q}\}$  $(1\le q\le r)$ is a set of different tubes
    and $S=\{\si_q\}$ $(1\le q\le r)$ is a set of costandard sequences $\si_q$ for each tube $T^{f_q}$.  We call such data \emph{special}
    if at least one of the tubes $\kT^{f_q}$ is special. Special costandard data are said to be even  or odd if all costandard sequences $\si_q$
    for special tubes $\kT^{f_q}$ are so. We set $N_\De=\bop_{q=1}^rM^{f_q}_{\si_q}$ and $z_{\De,n}=\sum_{q=1}^s z^{f_q}_{\si_q,n}$.
   In the latter definition we suppose that, if $\De$ is special, it is even if $n$ is even and it is odd if $n$ is odd (otherwise the elements
    $z^{f_q}_{\si,q,n}$ and hence $z_{\De,n}$ are not defined).
 
    \vskip0.5em\noindent
       If the tube $\kT^\8$ occurs in $\Si$, say, $f_k=\8$, 
  we denote by $z_{\De,n}^\8=z^\8_{\si_k,n}$ and by $z_{\De,n}^0$ the rest of the sum defining $z_{\De,n}$. 
  Of course, it is possible that $z_{\De,n}^\8=0$ or $z_{\De,n}^0=0$.
 \end{enumerate}
   \end{defin}
   
   \begin{theorem}\label{dcohom} 
  Let $N=DM$, where $M$ is a regular $R$-lattice, $\eps\in H^n(K,N)$ $(n>0)$. There are costandard data $\De$ and a decomposition
  $\theta:N\ito N_0\+N_\De$
   such that the projection of $\theta(\eps)$ onto $H^n(K,N_0)$ is zero and 
  the projection of  $\theta(\eps)$ onto $H^n(K,N_\De)$ equals $\eta(z_{\De,n})$
  \emph{(see page~\pageref{eta} for the definition of $\eta$).}
  
  \emph{If $\eps=0$, $M_\De=0$.}
 \end{theorem}

 \begin{corol}\label{dorbit} 
  Let $N$ be an indecomposable regular colattice. Consider the chain \eqref{dchain} of its submodules and denote by
  $H^n_k(K,N)$ the image in $H^n(K,N)$ of $H^n(K,N_k)$. Then the orbits of $\aut_KN$ on $H^n(K,N)\ (n>0)$
  are $H^n_k(K,N)\setminus H^n_{k-1}(K,N)\ (0< k\le m)$ and $\{0\}$.
 \end{corol}
 
 \section{Applications}

 \subsection{Crystallographic groups}
 \label{crystal} 

 Recall that a \emph{\crg} $G$ is a discontinuous group of isometries of an Euclidean space having a compact fundamental domain \cite{Szczep}.
 Equivalently, $G$ contains a maximal commutative subgroup $M$ of finite index, which is normal and is a free abelian group of finite rank. 
 Then the group $\Ga=\Ga/M$ acts on $M$ by the rule ${^gv}=\bar{g}v\bar{g}^{-1}$, where $\bar{g}$ is a preimage of $g$ in $G$, 
 and $G$ is given by a class $\eps\in H^2(\Ga,M)$. One easily sees that actually $M$ is a unique maximal abelian subgroup of $G$ of finite index. 
 We call the group $\Ga$ the \emph{top} and the $\Ga$-module $M$ the \emph{base} of the \crg\ $G$.  
 If $\vi:G\ito G'$, where $G'$ is another \crg, then $M'=\vi(M)$ is the maximal commutative subgroup of $G'$, i.e. the base of $G'$.
 Hence $\Ga'=G'/M'$ is the top of $G'$ and we have a commutative diagram
 \[
  \xymatrix{ 1 \ar[r] & M \ar[r] \ar[d]_{\theta} & G \ar[r] \ar[d]_\vi & \Ga \ar[r] \ar[d]_{\psi} & 1 \\
  				  1 \ar[r] & M' \ar[r]  & G' \ar[r]  & \Ga' \ar[r]  & 1	  } ,
 \]
 where $\theta$ and $\psi$ are isomorphisms. Let ${^\psi\!M'}$ be the group $M'$ considered as $G$-module by the rule ${^gu}={^{\psi(g)}u}$ for
 $g\in G,\,u\in M'$. Then $\theta$ is an isomorphism $M\ito {^\psi\!M'}$ and the cohomology class defining the group $G'$ is $\theta \eps\psi^{-1}$.
 Therefore, isomorphism classes of crystallographic groups with the top $\Ga$ and the base $M$ are in \oc\ with the orbits of the action of the group
 $\Aut G\xx\Aut_GM$ on $H^2(G,M)$.
 
 Therefore, Theorem~\ref{cohom} implies a classification result for \crg{s} with the Kleinian top and regular base.
 
 \begin{defin}\label{cgrs} 
  Let $\De$ be standard data, even if they are special. We call the group $\Cr(\De)$ that is the extension of $K$ with the kernel $M$ corresponding to
  the cohomology class $\eps=\xi(e_{\De,2})$ a \emph{standard \crg}.
  
  \smallskip
  Note that $\Cr(\De)$ is generated by the group $M$ and two elements $\ba$ and $\bb$ subject to the relations
  \begin{align*}
   & \ba w=({^aw})\ba \text{\ \emph{for every}\ } w\in M,\\
   &\bb w=({^bw})\bb \text{\ \emph{for every}\ } w\in M,\\
   & \ba\bb=\bb\ba,\\
   & \ba^2=e_{\De,2}^0,\\
   & \bb^2=e_{\De,2}^\8.
  \end{align*}
 \end{defin}
 
 \begin{theorem}\label{cryst}
  Let $G$ be a \crg\ with the Kleinian top $K$ such that its base $M$ is a regular $K$-lattice.
 \begin{enumerate}
   \item   There are standard data $\De$, 
  a direct decomposition $M\simeq M_0\+M_\De$
  and a semidirect decomposition $G\simeq M_0\rtimes \Cr(\De)$, where $\Cr(\De)$ acts on $M_0$ as
  its quotient $\Cr(\De)/M_\De\simeq K$.
  
  \item  If $G\simeq M'_0\rtimes \Cr(\De')$ is another such decomposition,  there is an automorphism $\psi$ of the group $K$
  such that $M'_0\simeq{^\psi\!M_0}$ and $\De'=\psi\De$.
   \end{enumerate}  
 
 \end{theorem}

 \begin{remk}
  $G$ is crystallographic \iff $M_{\al\be}\ne0$ for at least two of the pairs $(+-),(-+),(--)$. For a regular $K$-lattice $M$ it means that it is not
  a multiple of some lattice $T^{\la1}_1\ (\la\in\{0,1,\8\})$.
 \end{remk}
 
  \subsection{Chernikov groups}
  \label{chernikov} 
  
  Recall that a \emph{\chg} is a locally finite group with minimality condition on subgroups \cite{cher}. Such a group $G$ has a maximal divisible subgroup
  $N$ which is a finite direct sum of quasicyclic groups and $N$ is normal in $G$ with the finite quotient $\Ga=G/N$. We consider the case when
   $G$ is a $2$-group. Then $N$ is a direct sum of groups $\mD$
  of type $2^\8$ and $\Ga$ is a finite $2$-group. It is known that $\End\mD\simeq\mZ_2$. 
  Therefore, if $N=\mD^d$, then $\aut_\mZ N\simeq\GL(d,\mZ_2)$. Hence $N\simeq DM$ for some $\Ga$-lattice $M$.
  The group $\Ga$ and the  $G$-module $N$ are defined up to an isomorphism. We call $\Ga$  the \emph{top} and $N$ the \emph{base} of the \chg\ $G$. 
  Again, the isomorphism classes of Chernikov groups with the top $\Ga$ and the base $N$ are in \oc\ with the orbits of the group
  $\aut \Ga\xx\aut_\Ga N$ on the cohomology group $H^2(\Ga,N)$.
  
  Theorem~\ref{dcohom} implies the following description of Chernikov groups with the Kleinian top and regular bottom.
  
   \begin{defin}\label{chrs} 
  Let $\De$ be costandard data, even if they are special. We call the group $\Ch(\De)$ that is the extension of $K$ with the kernel $N$ corresponding to
  the cohomology class $\eps=\eta(z_{\De,2})$ a \emph{standard \chg}.
  
  \smallskip
  Note that $\Ch(\De)$ is generated by the group $N$ and two elements $\ba$ and $\bb$ subject to the relations
  \begin{align*}
   & \ba w=({^aw})\ba \text{\ \emph{for every}\ } w\in N,\\
   &\bb w=({^bw})\bb \text{\ \emph{for every}\ } w\in N,\\
   & \ba\bb=\bb\ba,\\
   & \ba^2=z^0_{\De,2}\\
   & \bb^2=z^\8_{\De,2}.
  \end{align*}
 \end{defin}
 
 \begin{theorem}\label{chern}
  Let $G$ be a \chg\ with the Kleinian top $K$ such that its base $N$ is a regular $K$-colattice. 
\begin{enumerate}
  \item  There are costandard data $\De$, a direct decomposition $N=N_0\+N_\De$
  and a semidirect decomposition $G\simeq N_0\rtimes \Ch(\De)$, where $\Ch(\De)$ acts on $N_0$ as
  its quotient $\Ch(\De)/N_\De\simeq K$.
  
  \item  If $G\simeq N'_0\rtimes \Ch(\De')$ is another such decomposition,  there is an automorphism $\psi$ of the group $K$
  such that $N'_0\simeq{^\psi\!N_0}$ and $\De'=\psi\De$.
  \end{enumerate}  
 \end{theorem}

 Note that there is an amazing parallelism between chrystallographic and Chernikov groups. Perhaps, it is of general nature, and is worth to
 study.
 
 \bigskip
 {\small
 \textbf{Aknowledgement}
 
 \medskip
 The authors are greatful to the referee of the paper for useful notices that were used to essentially improve the exposition.}

\end{document}